\documentclass[reqno,12pt]{amsart}

\usepackage[utf8]{inputenc}
\usepackage{microtype}
\usepackage{amsfonts}
\usepackage{amsmath}
\usepackage{graphicx}
\usepackage{hyperref}
\usepackage{color}

\usepackage{amssymb}
\usepackage{enumitem}
\usepackage{stmaryrd}

\usepackage{a4wide}

\usepackage{empheq}

\newtheorem{theorem}{Theorem}[section]
\newtheorem{lemma}[theorem]{Lemma}
\newtheorem{proposition}[theorem]{Proposition}
\newtheorem{corollary}[theorem]{Corollary}
\newtheorem{observation}[theorem]{Observation}

\newtheorem*{main:main_thm}{Main Theorem}
\newtheorem*{introcor}{Corollary}

\theoremstyle{definition}
\newtheorem{definition}[theorem]{Definition}
\newtheorem{remark}[theorem]{Remark}
\newtheorem{example}[theorem]{Example}

\newcommand{\Z}{\mathbb{Z}}
\newcommand{\N}{\mathbb{N}}

\newcommand{\R}{\mathbb{R}}

\newcommand{\tree}{\mathcal{T}}

\DeclareMathOperator{\id}{id}

\DeclareMathOperator{\Aut}{Aut}
\DeclareMathOperator{\AAut}{AAut}
\DeclareMathOperator{\Grig}{Grig}
\renewcommand{\deg}{\operatorname{deg}}
\renewcommand{\mod}{\mathrel{\operatorname{mod}}}

\newcommand{\F}{\mathrm{F}}
\newcommand{\FP}{\mathrm{FP}}
\newcommand{\FF}{\mathbb{F}}

\newcommand{\calC}{\mathcal{C}}
\newcommand{\calD}{\mathcal{D}}
\newcommand{\calE}{\mathcal{E}}
\newcommand{\calF}{\mathcal{F}}
\newcommand{\calN}{\mathcal{N}}
\newcommand{\calW}{\mathcal{W}}
\newcommand{\Ob}{\operatorname{Ob}}
\newcommand{\head}{\operatorname{head}}

\newcommand{\abs}[1]{\lvert #1 \rvert}
\newcommand{\ceil}[1]{\lceil #1 \rceil}
\newcommand{\floor}[1]{\lfloor #1 \rfloor}
\newcommand{\dfloor}[1]{\left\lfloor #1 \right\rfloor}

\newcommand{\defeq}{\mathbin{\vcentcolon =}}

\newcommand{\bfG}{\mathbf{G}}
\newcommand{\AGL}{\operatorname{AGL}}
\newcommand{\PU}{\operatorname{PU}}
\newcommand{\PGL}{\operatorname{PGL}}
\newcommand{\llb}{(\mkern-2mu(}
\newcommand{\rrb}{)\mkern-2mu)}
\newcommand{\lls}{[\mkern-2mu[}
\newcommand{\rrs}{]\mkern-2mu]}

\newcommand{\lseries}[1]{\llb #1 \rrb}
\newcommand{\pseries}[1]{\lls #1\rrs}
\newcommand{\calO}{\mathcal{O}}
\newcommand{\frakm}{\mathfrak{m}}

\begin{document}

\title{Simple groups separated by finiteness properties}
\date{\today}
\subjclass[2010]{Primary 20E32; 
                 Secondary 57M07, 
			20F65,   
			20E08} 

\keywords{Simple group, finiteness properties, quasi-isometry, tree automorphism, almost-automorphism, R\"over--Nekrashevych group, self-similar group, Thompson's group}

\author[R.~Skipper]{Rachel Skipper}
\address{Mathematics Institute, University of G{\"o}ttingen, Bunsenstrasse 3-5, 37073 G{\"o}ttingen, Germany}
\email{skipper.rachel.k@gmail.com}

\author[S.~Witzel]{Stefan Witzel}
\address{Department of Mathematics, Bielefeld University, PO Box 100131, 33501 Bielefeld, Germany}
\email{switzel@math.uni-bielefeld.de}

\author[M.~C.~B.~Zaremsky]{Matthew C.~B.~Zaremsky}
\address{Department of Mathematics and Statistics, University at Albany (SUNY), Albany, NY 12222}
\email{mzaremsky@albany.edu}

\thanks{The first author was partially supported by a grant from the Simons Foundation (\#245855 to Marcin Mazur). The second author is funded through the DFG project WI 4079-2/2.}

\begin{abstract}
We show that for every positive integer $n$ there exists a simple group that is of type $\F_{n-1}$ but not of type $\F_n$. For $n\ge 3$ these groups are the first known examples of this kind. They also provide infinitely many quasi-isometry classes of finitely presented simple groups. The only previously known infinite family of such classes, due to Caprace--Rémy, consists of non-affine Kac--Moody groups over finite fields. Our examples arise from R\"over--Nekrashevych groups, and contain free abelian groups of infinite rank.
\end{abstract}

\maketitle
\thispagestyle{empty}

\section*{Introduction}\label{sec:intro}

A group is of type $\F_n$ if it admits a classifying space with a compact $n$-skeleton. These topological finiteness properties generalize being finitely generated (type $\F_1$) and being finitely presented (type $\F_2$), and are quasi-isometry invariants. In this article we prove:

\begin{main:main_thm}
 For every positive integer $n$ there exists a simple group that is of type $\F_{n-1}$ but not of type $\F_n$.
\end{main:main_thm}

This result is new for every $n \ge 3$; see Theorem~\ref{thm:main_thm_precise} for a precise version. We say that two groups are \emph{separated} by finiteness properties if, for some $n$, one of them is of type $\F_n$ and the other is not. In this sense the Main Theorem provides infinitely many simple groups that are pairwise separated by finiteness properties. Our examples arise from certain groups of homeomorphisms of Cantor spaces, called R\"over--Nekrashevych groups. Roughly speaking, a R\"over--Nekrashevych group is built out of a Higman--Thompson group and a self-similar group, and we prove the Main Theorem by showing that, under certain conditions, the R\"over--Nekrashevych group inherits (virtual) simplicity from the Higman--Thompson group and finiteness properties from the self-similar group. The main novelty is that the groups can be constructed to not be of type $\F_n$.

One of our primary motivations is to distinguish infinite simple groups:

\begin{introcor}
There are infinitely many quasi-isometry classes of groups that are finitely presented, simple and contain free abelian subgroups of infinite rank.
\end{introcor}

While finite simple groups have been classified, in one of the largest collective efforts in pure mathematics, comparatively little is known about the possible variety among infinite simple groups. To make the problem approachable, it is natural to restrict to the countable family of finitely presented groups and, following Gromov's insight, to study them up to quasi-isometry (note that all finite groups form a single quasi-isometry class). So far there was only one known infinite family of quasi-isometry classes of finitely presented simple groups: Caprace and Rémy showed that non-affine Kac--Moody groups provide such a family \cite{caprace09,caprace10}. Rémy asked whether infinitely many quasi-isometry classes of finitely presented simple groups could also be found in the realm of generalized Thompson groups, and we answer this question affirmatively. Since every Kac--Moody group has finite asymptotic dimension, none of our examples is quasi-isometric to any of them.

It is unclear whether the Main Theorem, at least restricted to even $n$, could also be proved using Kac--Moody groups. An irreducible Kac--Moody group is virtually simple if it is not of affine type \cite{caprace09}. However, the precise finiteness properties of infinite Kac--Moody groups are known only in the affine case \cite{witzel14phd}. A $2$-spherical Kac--Moody group over a large enough finite field is finitely presented \cite{abramenko97}. On the other hand unless it is finite it is not of type $\F_\infty$, being a non-uniform lattice \cite{gandini12,kropholler93,kropholler98}.

Infinite simple groups are easy to obtain if one does not require finite generation. The existence of a finitely generated infinite simple group was shown by Higman \cite{higman51}. Uncountably many such groups were constructed shortly after by Camm \cite{camm53}. The first known finitely presented infinite simple groups were Thompson's groups $T$ and $V$ (described in unpublished notes). They were extended to an infinite family of examples by Higman \cite{higman74}. Brown \cite{brown87} later showed that all these finitely presented examples are of type $\F_\infty$, following Brown and Geoghegan's proof \cite{brown84} that Thompson's group $F$ (which is not simple) is of type $\F_\infty$. Since then other simple variations of Thompson's groups have been constructed, notably the Brin--Thompson groups $nV$ \cite{brin04,brin10}. However, in all cases established so far, these groups are of type $\F_\infty$ \cite{fluch13,belk16,skipper17} or not finitely presented \cite{witzel16}.

Another class of finitely presented simple groups consists of the examples by Burger and Mozes \cite{burger00}. While these are interesting for various reasons, they only form a single quasi-isometry class, because all of them are uniform lattices on a product of two trees.

There has also been a great deal of interest in studying groups separated by finiteness properties. The first group of type $\F_2$ but of type not $\F_3$ was constructed by Stallings \cite{stallings63}. It was generalized by Bieri \cite{bieri76} to an infinite sequence of groups separated by finiteness properties, and then to a very general construction by Bestvina and Brady \cite{bestvina97}. This program was eventually finished by Meier, Meinert, and VanWyk \cite{meier98}. All the groups in this vein are subgroups of right-angled Artin groups and hence are residually finite. Another large source of groups separated by finiteness properties is $S$-arithmetic groups \cite{bux04,abels87,bux13} and variations thereof \cite{bux10}, which are also residually finite. Examples of non-residually finite groups separated by finiteness properties include Houghton's groups \cite[Section~5]{brown87}, subgroups of Thompson's group $F$ constructed by Bieri--Geoghegan--Kochloukova \cite{bieri10}, and certain generalized Thompson groups considered in \cite{witzel14}. In all these cases the abelianization is infinite.

The paper is organized as follows. We recall some background on finiteness properties and quasi-isometries in Section~\ref{sec:fin_props_QI}, and on self-similar groups and R\"over--Nekrashevych groups in Section~\ref{sec:gps}. The following sections are largely independent and contain the proofs that certain R\"over--Nekrashevych groups are virtually simple (Section~\ref{sec:simple}), of type $\F_{n-1}$ (Section~\ref{sec:positive}), and not of type $\F_n$ (Section~\ref{sec:negative}), assuming the existence of self-similar groups with appropriate properties. Such self-similar actions of groups separated by finiteness properties are constructed in Section~\ref{sec:self_similar_fin_props}. Finally, in Section~\ref{sec:proof} we prove the Main Theorem, provide some examples, and discuss the (non-)relationship with Kac--Moody groups.

\subsection*{Acknowledgments} We are grateful to Dessislava Kochloukova and Said Sidki for sharing a preprint of \cite{kochloukova17} with us and to Bertrand Rémy for asking the question that motivated our search for simple groups separated by finiteness properties. We also thank Matt Brin, Pierre-Emmanuel Caprace, Eduard Schesler, and Marco Varisco for many helpful comments.

\section{Finiteness properties and quasi-isometries}\label{sec:fin_props_QI}

A \emph{classifying space} for a group $G$ is a connected CW-complex $X$ with fundamental group $G$ and contractible universal cover. A group is of \emph{type $F_n$} if it admits a classifying space with compact $n$-skeleton. Being of type $\F_1$ is equivalent to being finitely generated, and being of type $\F_2$ is equivalent to being finitely presented. A group is of \emph{type $F_\infty$} if it is of type $\F_n$ for all $n$. These topological finiteness properties $\F_n$ also have slightly weaker homological analogues $\FP_n$. Rather than defining these we summarize the relationship: type $\F_n$ implies type $\FP_n$; type $\FP_1$ is equivalent to type $\F_1$; type $\F_2$ and type $\FP_n$ imply type $\F_n$; type $\FP_2$ does not imply $\F_2$, by a famous result due to Bestvina and Brady \cite{bestvina97}, recently expanded by Leary \cite{leary15,leary16}.

Topological and homological finiteness properties are quasi-isometry invariants. This was shown by Alonso \cite{alonso94} applying Brown's criterion \cite{brown87} to Rips complexes. We will need a more precise result from \cite{alonso94} which we describe now.

\begin{definition}[$(C,D)$-Lipschitz, quasi-isometric]
A function $f\colon X\to Y$ between metric spaces $(X,d_X)$ and $(Y,d_Y)$ is \emph{$(C,D)$-Lipschitz} (for $C\ge 1$ and $D\ge 0$) if
\[
d_Y(f(x),f(x'))\le Cd_X(x,x')+D
\]
for all $x,x'\in X$. The metric spaces $X$ and $Y$ are \emph{quasi-isometric} if there exist functions $f\colon X\to Y$ and $f'\colon Y \to X$ and constants $C$ and $D$ such that $f$ and $f'$ are $(C,D)$-Lipschitz and for all $x\in X$ and $y\in Y$ we have
\[
d_X(f'f(x),x)\leq D \quad \text{and} \quad d_Y(ff'(y),y) \leq D\text{.}
\]
\end{definition}

\begin{definition}[Quasi-retract]
Let $H$ and $Q$ be finitely generated groups, equipped with word metrics $d_H$ and $d_Q$ respectively. If there exist $(C,D)$-Lipschitz functions $r\colon H\to Q$ and $\iota\colon Q\to H$ such that $d_Q(r\circ\iota(x),x)\le D$ for all $x\in Q$ then we call $Q$ a \emph{quasi-retract} of $H$, and $r$ a \emph{quasi-retraction}.
\end{definition}

Recall that $Q$ is called a \emph{retract} of $H$ if $\iota$ and $r$ are homomorphisms and $r\circ\iota=\id_Q$, so quasi-retracts are a natural geometric generalization of retracts. Finiteness properties are preserved under passing to retracts (see Remark~\ref{rmk:retract}), and Alonso proved that they are even preserved under passing to quasi-retracts:

\begin{theorem}[{\cite[Theorem~8]{alonso94}\label{thm:alonso}}]
 Let $H$ and $Q$ be finitely generated groups such that $Q$ is a quasi-retract of $H$. If $H$ is of type $\F_n$ (or $\FP_n$) then so is $Q$.
\end{theorem}

\begin{corollary}[{\cite[Corollary~9]{alonso94}\label{cor:alonso}}]
Let $H$ and $Q$ be finitely generated, quasi-isometric groups (under the word metric). Then $H$ is of type $\F_n$ (or $\FP_n$) if and only if $Q$ is of type $\F_n$ (or $\FP_n$). In particular, finiteness properties are invariant under quasi-isometry.
\end{corollary}

\begin{remark}
The main difficulty in proving the Main Theorem is to show that the groups are not of type $\F_n$. One possibility for proving that a group $H$ is not of type $\F_n$ is to show that it has a retract $Q$ that is not of type $\F_n$. For a simple group this is clearly not an option. Theorem~\ref{thm:alonso} says that it suffices for $Q$ to be a quasi-retract, which, as we will see in Section~\ref{sec:negative}, can happen even for simple groups.
\end{remark}

\section{Self-similar groups and R\"over--Nekrashevych groups}\label{sec:gps}

Let $X$ be a set with $d\ge 2$ elements, called an \emph{alphabet}, and let $X^*$ be the set of finite words in $X$, including the empty word $\varnothing$. Then $X^*$ can be identified with the vertex set of the infinite rooted $d$-ary tree $\tree_d$. Two words $u,v\in X^*$ share an edge in $\tree_d$ precisely when $u=vx$ or $v=ux$ for some $x\in X$. The root corresponds to $\varnothing$, the only vertex of degree $d$. We fix a total order on $X$ and order $X^*$ lexicographically.

In this section we recall our groups of interest, namely self-similar groups and R\"over--Nekrashevych groups. More background can be found for example in \cite{nekrashevych05}.

\subsection{Self-similar groups}\label{sec:self_sim}

An \emph{automorphism} of $\tree_d$ will always mean a bijection from $\tree_d$ to $\tree_d$ taking vertices to vertices, edges to edges, and preserving incidence. We denote by $\Aut(\tree_d)$ the group of automorphisms of $\tree_d$. Note that every automorphism fixes the root $\varnothing$.

If we identify $X$ with $\{1,\dots,d\}$ then, since automorphisms of $\tree_d$ preserve length of words in $X^*$, any automorphism of $\tree_d$ induces a permutation of $\{1,\dots,d\}$, i.e., an element of the symmetric group $S_d$. Hence we get a map $\rho\colon \Aut(\tree_d) \to S_d$. Moreover, since for any $x\in X$ the subgraph of $\tree_d$ spanned by the vertex set $\{xw\mid w\in X^*\}$ is canonically isomorphic to $\tree_d$ via $xw\leftrightarrow w$, the kernel of $\rho$ is a direct product of $d$ copies of $\Aut(\tree_d)$. The action of $S_d$ naturally permutes these copies, and we conclude that $\Aut(\tree_d) \cong S_d \wr \Aut(\tree_d)$.

\begin{definition}[Wreath recursion]
 Let $\psi \colon \Aut(\tree_d) \to S_d \wr \Aut(\tree_d)$ be the above isomorphism. For any $f\in \Aut(\tree_d)$ if $\psi(f)=\rho(f)(f_1,\dots,f_d)$ then we call $\rho(f)(f_1,\dots,f_d)$ the \emph{wreath recursion} of $f$. As is standard, we will often abuse notation and just write $f=\rho(f)(f_1,\dots,f_d)$.
\end{definition}

Given $f\in \Aut(\tree_d)$ we will be interested in the states of $f$, meaning the elements that can appear when applying iterated wreath recursions to $f$. More rigorously, we have the following definition.

\begin{definition}[States]
 For $f\in\Aut(\tree_d)$, the \emph{level-$1$ states} of $f$ are the elements $f_1,\dots,f_d$ appearing in the wreath recursion $f=\rho(f)(f_1,\dots,f_d)$. The \emph{states} of $f$ are the elements of the smallest subset of $\Aut(\tree_d)$ that contains $f$ and is closed under taking level-$1$ states.
\end{definition}

Iterating the wreath recursion $n$ times produces a permutation of the vertices on the $n$th level and a collection of states on the $n$th level.

\begin{definition}[Self-similar]
 Let $S\subseteq \Aut(\tree_d)$. We call $S$ \emph{self-similar} if for all $s\in S$, every state of $s$ lies in $S$. We will also refer to a group $G$ as having a \emph{faithful self-similar action} on $\tree_d$ if $G$ acts faithfully on $\tree_d$ and its resulting image in $\Aut(\tree_d)$ is self-similar.
\end{definition}

Self-similar groups are sometimes called \emph{state-closed} groups in the literature.

\begin{definition}[Finite-state]
An element $s \in \Aut(\tree_d)$ is \emph{finite-state} if its set of states is finite. A set $S \subseteq \Aut(\tree_d)$ is finite-state if each of its elements is. (In the literature this is sometimes only defined when $S$ is self-similar.) We will also refer to a group $G$ as having a \emph{faithful finite-state action} on $\tree_d$ if $G$ acts faithfully on $\tree_d$ and its resulting image in $\Aut(\tree_d)$ is finite-state.
\end{definition}

Finite-state self-similar groups are often called \emph{automata groups} in the literature.

Note that the wreath recursion of a product, say $fg$ for $f,g\in \Aut(\tree_d)$, is
\[
fg=\rho(f)\rho(g)(f_{\rho(g)(1)}g_1,\dots,f_{\rho(g)(d)}g_d)
\]
where $f=\rho(f)(f_1,\dots,f_d)$ and $g=\rho(g)(g_1,\dots,g_d)$ are the wreath recursions of $f$ and $g$. Iterating this, we see:

\begin{lemma}
\label{lem:product_of_finite_state}
If $f,g \in \Aut(\tree_d)$ are finite-state then so is $fg$. In particular, $G \le \Aut(\tree_d)$ is finite-state if it admits a finite-state generating set.
\end{lemma}

\subsection{R\"over--Nekrashevych groups}\label{sec:nekr}

Let $X^\omega$ be the set of infinite words in $X$, so $X^\omega$ can be identified with the boundary $\partial\tree_d$, which is a $d$-ary Cantor set. For any $u\in X^*$ define the \emph{cone} of $u$ to be $C(u)\defeq\{uw\mid w\in X^\omega\} \subseteq X^\omega$. Call a finite subtree $T$ of $\tree_d$ \emph{rooted} if it contains the root of $\tree_d$. Call $T$ \emph{complete} if whenever $ux\in T$ for $u\in X^*$ and $x\in X$, in fact $uy\in T$ for all $y\in X$. A vertex $u$ of $T$ is a \emph{leaf} of $T$ if no $ux$ for $x\in X$ lies in $T$.

\begin{definition}[Almost-automorphism]
 An \emph{almost-automorphism} of $\tree_d$ is a homeomorphism of $\partial\tree_d \to \partial\tree_d$ that can be obtained as follows. Take two finite complete rooted subtrees $T_-$ and $T_+$ of $\tree_d$ with the same number of leaves, say $n$. Let $u_1,\dots,u_n$ be the leaves of $T_-$ in order and let $v_1,\dots,v_n$ be the leaves of $T_+$ in order. Also let $f_1,\dots,f_n\in\Aut(\tree_d)$ and $\sigma\in S_n$. Note that the cones $C(u_1),\dots,C(u_n)$ partition $\partial\tree_d$, as do the cones $C(v_1),\dots,C(v_n)$. We collect all these data into a triple $(T_-,\sigma(f_1,\dots,f_n),T_+)$. Now define the almost-automorphism $[T_-,\sigma(f_1,\dots,f_n),T_+]$ to be the self-homeomorphism of $\partial\tree_d$ sending $v_i w$ to $u_{\sigma(i)} f_i(w)$ for all $1\le i\le n$. We denote by $\AAut(\tree_d)$ the group of all almost-automorphisms of $\tree_d$.
\end{definition}

\begin{definition}[Higman--Thompson group]
 The \emph{Higman--Thompson group} $V_d$ is the subgroup of $\AAut(\tree_d)$ consisting of homeomorphisms $[T_-,\sigma,T_+]$.
\end{definition}

The following is well known, and is the reason that our Main Theorem implies the Corollary in the introduction:

\begin{lemma}\cite[Theorem~4.8]{cannon96}\label{lem:infinite_rank_abelian_subgroup}
The group $V_d$ contains free abelian subgroups of infinite rank.
\end{lemma}

\begin{definition}[R\"over--Nekrashevych group]
 Let $G\le \Aut(\tree_d)$ be self-similar. The \emph{R\"over--Nekrashevych group} $V_d(G)$ is the subgroup of $\AAut(\tree_d)$ consisting of homeomorphisms $[T_-,\sigma(g_1,\dots,g_n),T_+]$ for $g_1,\dots,g_n\in G$.
\end{definition}

Note that $V_d(G)$ really is a subgroup of $\AAut(\tree_d)$, since self-similarity ensures that it is closed under multiplication. R\"over--Nekrashevych groups were introduced in this degree of generality in \cite{nekrashevych04}. The first such group considered was the \emph{R\"over group} \cite{roever99}, which is $V_2(\Grig)$ for $\Grig$ the Grigorchuk group. The R\"over group is simple and of type $\F_\infty$ \cite{belk16}.

A given almost-automorphism is represented by various triples, and the rest of this section will be devoted to understanding this non-uniqueness. For example $[T,(\id,\dots,\id),T]$ is the identity for any $T$. More generally, two triples
\[
(T_-,\sigma(f_1,\dots,f_n),T_+)\quad\text{and}\quad (U_-,\tau(g_1,\dots,g_n),U_+)
\]
describe the same almost-automorphism
\[
[T_-,\sigma(f_1,\dots,f_n),T_+]=[U_-,\tau(g_1,\dots,g_n),U_+]
\]
if both can be expanded in the following sense to produce the same triple.\\
Let $(T_-,\sigma(f_1,\dots,f_n),T_+)$ be a triple representing an element of $\AAut(\tree_d)$. Say the leaves of $T_-$ are $u_1,\dots,u_n$ and the leaves of $T_+$ are $v_1,\dots,v_n$ (in order). Now let $T_+'$ be a finite complete rooted subtree of $\tree_d$ containing $T_+$, say with leaves $v_1',\dots,v_{n'}'$, so the partition of $\partial\tree_d$ by the cones $C(v_1'),\dots,C(v_{n'}')$ is a refinement of the partition by the cones $C(v_1),\dots,C(v_n)$. Looking at the image of these cones under $[T_-,\sigma(f_1,\dots,f_n),T_+]$ we get a partition refining $C(u_1),\dots,C(u_n)$, of the form $C(u_1'),\dots,C(u_{n'}')$ for $u_i'$ the leaves of some tree $T_-'$. Now by iterated applications of wreath recursions, we can view $\sigma(f_1,\dots,f_n)$ as some $\sigma'(f_1',\dots,f_{n'}')$, acting on $C(v_i')$ by $f_i'$ and then taking it to $C(u_{\sigma'(i)}')$.

\begin{definition}[Expansion]
With the above setup, we call $(T_-',\sigma'(f_1',\dots,f_{n'}'),T_+')$ an \emph{expansion} of $(T_-,\sigma(f_1,\dots,f_n),T_+)$.
\end{definition}

The following confluence result provides a converse to what was said above (cf.\ \cite[Lemma~2.3]{leboudec17}).

\begin{lemma}\label{lem:iff_exp}
Two triples represent the same almost-automorphism if and only if they have a common expansion.
\end{lemma}

\begin{proof}
Note that if $f = [T_-, \sigma(f_1,\ldots,f_n),T_+]$ is an almost-automorphism, the entries $\sigma(f_1,\ldots,f_n)$ and $T_-$ are uniquely determined by $f$ and $T_+$. The claim therefore follows from the fact that for any two finite rooted trees there is a finite rooted tree containing both.
\end{proof}

Note that every expansion can be obtained by a sequence of expansions of a single leaf, which makes the following lemma particularly useful. By a \emph{$d$-caret} we mean the finite rooted complete subtree of $\tree_d$ with $d$ leaves.

\begin{lemma}[One leaf expansion]\label{lem:expand_one}
 Let $(T_-,\sigma(f_1,\dots,f_n),T_+)$ be a triple representing an element of $\AAut(\tree_d)$ and let $1 \le k \le n$. Let $T_+'$ be the result of adding a single $d$-caret to the $k$th leaf of $T_+$. Then the corresponding expansion of $(T_-,\sigma(f_1,\dots,f_n),T_+)$ is
 \[
 (T_-',\sigma'(f_1,\dots,f_{k-1},f_k^1,\dots,f_k^d,f_{k+1},\dots,f_n),T_+')\text{,}
 \]
where $f_k=\rho(f_k)(f_k^1,\dots,f_k^d)$ is the wreath recursion of $f_k$, $T_-'$ is some tree and $\sigma'$ is some element of $S_{n+d-1}$.
\end{lemma}

\begin{proof}
 Let $v_1,\dots,v_n$ be the leaves of $T_+$ (in order). Write $X=\{x_1,\dots,x_d\}$, so the leaves of $T_+'$ are
\[
v_1,\dots,v_{k-1},v_k x_1,\dots,v_k x_d,v_{k+1},\dots,v_n\text{.}
\]
Now for any $1\le i\le d$ and any $w\in X^*$ we have $f_k(x_i w) = x_{\rho(f_k)(i)} f_k^i(w)$. This implies that $\sigma(f_1,\dots,f_n)$ acts on the partition by cones of vertices of $T_+'$ as
\[
\sigma'(f_1,\dots,f_{k-1},f_k^1,\dots,f_k^d,f_{k+1},\dots,f_n)
\]
for some $\sigma'$, as desired.
\end{proof}

The precise description of $T_-'$ and $\sigma'$ is the content of a \emph{$d$-ary cloning system}, as in \cite{skipper17}, and will not be needed here.

\section{Virtual simplicity}\label{sec:simple}

We are mainly interested in R\"over--Nekrashevych groups that are virtually simple. It turns out this is the case as soon as the abelianization is finite:

\begin{theorem}[Simple commutator {\cite[Theorem~4.7]{nekrashevych17}}]\label{thm:simple}
 For any self-similar $G\le \Aut(\tree_d)$, the commutator subgroup $V_d(G)'$ is simple.
\end{theorem}

See \cite[Theorem~9.14]{nekrashevych04} for a description of the abelianization and commutator subgroup of $V_d(G)$. In order to ensure finite abelianization we introduce the following notion.

\begin{definition}[Coarsely diagonal]
 Let $G\le \Aut(\tree_d)$. We call $G$ \emph{coarsely diagonal} if for every $g\in G$ and every state $g'$ of $g$ the element $(g')^{-1}g$ has finite order. We will also refer to a group $G$ as having a \emph{faithful coarsely diagonal action} on $\tree_d$ if $G$ acts faithfully on $\tree_d$ and its resulting image in $\Aut(\tree_d)$ is coarsely diagonal.
\end{definition}

\begin{theorem}[Finite abelianization]\label{thm:fin_abln}
 Let $G\le \Aut(\tree_d)$ be a finitely generated, coarsely diagonal self-similar group. Then $V_d(G)$ has finite abelianization.
\end{theorem}

In the proof of the theorem we will use an explicit generating set for $V_d(G)$. For $u\in X^*$ let $\iota_u \colon G\to V_d(G)$ be the embedding sending $g$ to the automorphism that applies $g$ to the cone $C(u)$ and is trivial outside this cone. More rigorously, this means $\iota_u(g)(uw)=ug(w)$ for all $w\in X^*$ and $\iota_u(g)(w)=w$ whenever $u$ is not a prefix of $w$. If we identify $X$ with $\{1,\dots,d\}$ then we can refer to the maps $\iota_i$ for $1\le i\le d$.

\begin{lemma}[{\cite[Lemma~5.11]{nekrashevych17}\label{lem:nekr_gens}}]
 The group $V_d(G)$ is generated by $V_d\cup \iota_1(G)$.
\end{lemma}

\begin{proof}[Proof of Theorem~\ref{thm:fin_abln}]
 By Lemma~\ref{lem:nekr_gens}, $V_d(G)$ is generated by $V_d\cup \iota_1(G)$. Since $V_d$ is virtually simple \cite{higman74} its image in the abelianization of $V_d(G)$ is finite, so it suffices to show that the image of $\iota_1(G)$ in the abelianization of $V_d(G)$ is also finite. We will do this by proving that the image of $\iota_1(g)$ in the abelianization of $V_d(G)$, for any $g\in G$, is torsion (this gives us what we want since $G$ is finitely generated). Write $\overline{f}$ for the image of $f\in V_d(G)$ in the abelianization of $V_d(G)$. Note that for any $u,v\in X^*\setminus\{\varnothing\}$ and any $g\in G$, the elements $\iota_u(g)$ and $\iota_v(g)$ are conjugate in $V_d(G)$ by an element of $V_d$ that takes $uw$ to $vw$ for all $w\in X^*$, so $\overline{\iota_u(g)}=\overline{\iota_v(g)}$. Also note that if $g=\rho(g)(g_1,\dots,g_d)$ is the wreath recursion for $g$, then
\[
\iota_u(g)=\iota_u(\rho(g))\prod_{i=1}^d \iota_{ui}(g_i)\text{,}
\]
where by $\iota_u(\rho(g))$ we mean the automorphism acting as $\rho(g)$ on $C(u)$ and the identity elsewhere. In particular $\overline{\iota_u(g)}=\overline{\iota_u(\rho(g))}+\sum_{i=1}^d \overline{\iota_{ui}(g_i)}$. Using $u=1$ this gives us $\overline{\iota_1(g)}=\overline{\iota_1(\rho(g))}+\sum_{i=1}^d \overline{\iota_{1i}(g_i)}$, which equals $\overline{\iota_1(\rho(g))}+\sum_{i=1}^d \overline{\iota_1(g_i)}$. Now since $G$ is coarsely diagonal, for each $1\le i\le d$ there exists a torsion element $a_i$ of the abelianization of $V_d(G)$ such that $\overline{\iota_1(g_i)}=\overline{\iota_1(g)}+a_i$, so we get
\[
\overline{\iota_1(g)}=\overline{\iota_1(\rho(g))}+d(\overline{\iota_1(g)})+a_1+\cdots+a_d\text{.}
\]
This tells us that $(d-1)\overline{\iota_1(g)}=-\overline{\iota_1(\rho(g))}-a_1-\cdots-a_d$. Since $\overline{\iota_1(\rho(g))}$ and all the $a_i$ are torsion, we conclude that $\overline{\iota_1(g)}$ is torsion, and we are done.
\end{proof}

Another consequence of Lemma~\ref{lem:nekr_gens} is the following.

\begin{observation}\label{obs:fin_abln_easy}
 If $G$ has finite abelianization then so does $V_d(G)$.
\end{observation}

\begin{proof}
 As in the proof of Theorem~\ref{thm:fin_abln} it suffices to show that the image of $\iota_1(G)$ in the abelianization of $V_d(G)$ is finite, and since $G$ has finite abelianization this is immediate.
\end{proof}

\section{Proving type $\F_{n-1}$}\label{sec:positive}

In this section we prove that $V_d(G)$ is of type $\F_{n-1}$ if $G \le \Aut(\tree_d)$ is self-similar and of type $\F_{n-1}$. The proof is implicitly contained in \cite{skipper17}, and relies on established machinery that has been used to prove that various generalizations of Thompson groups are of type $\F_\infty$; see \cite{brown87,stein92,farley03,fluch13,bux16,belk16,martinez-perez16,thumann17}. We will use the following recent formalization of that machinery:

\begin{theorem}[{\cite[Theorem~3.12]{witzel17}}]\label{thm:ore_criterion}
Let $\calC$ be a right-Ore category and let $*$ be an object of $\calC$. Let $\calE$ be a locally finite left-Garside family of morphisms that is closed under taking factors. Let $\delta \colon \Ob(\calC) \to \N$ be a height function such that $\{x \in \Ob(\calC) \mid \delta(x) \le n\}$ is finite for every $n \in \N$. Assume
\begin{enumerate}[align=left,leftmargin=*,widest=(\textsc{stab})]
\item[\textsc{(stab)}] $\calC^\times(x,x)$ is of type $\F_n$ for all $x \in \Ob(\calC)$,
\item[\textsc{(lk)}] $\abs{E(x)}$ is $(n-1)$-connected for all $x$ with $\delta(x)$ beyond a fixed bound.
\end{enumerate}
Then $\pi_1(\calC,*)$ is of type $\F_n$.
\end{theorem}

We will explain the relevant notions as we go. We start by describing certain categories of homeomorphisms underlying R\"over--Nekrashevych groups. Let $C = X^\omega$ and for each $k \in \N$ let $C_k \defeq \{1,\ldots,k\} \times X^\omega$. Note that $C_k$ is the boundary of the infinite $d$-ary forest on $k$ roots. The set of vertices of this forest is $\{1,\ldots,k\} \times X^*$. All the categories in what follows will have as objects the spaces $C_k$ for $k \in \N$.

Notationally, we treat morphisms in a category as elements of the category. If $\calD$ is a category with objects $\Ob(\calD) = \{C_k \mid k \in \N\}$ we denote the set of morphisms from $C_\ell$ to $C_k$ in $\calD$ by $\calD(C_k,C_\ell)$. Similarly, $\calD(C_k,-)$ and $\calD(-,C_\ell)$ denote the set of morphisms with target $C_k$ respectively source $C_\ell$. We denote sets of invertible such morphisms by replacing $\calD$ with $\calD^\times$ in the notation. We always write morphisms as pointing left, which is convenient when evaluating compositions: if $A \stackrel{f}{\leftarrow} B \stackrel{g}{\leftarrow} C$ then $fg(c) \in A$ for $c \in C$. This is purely notational; all our morphisms represent bijective maps, and one could equivalently reverse all arrows and decorate all morphisms with an exponent $-1$ (cf.\ also \cite[Remark~1.1]{witzel17} and the preceding paragraphs).

A finite rooted complete subtree $T$ of $\tree_d$ with $m$ leaves $(u_1,\ldots,u_m)$ (in order) defines a homeomorphism $\varphi_T \colon C \leftarrow C_m, u_iw \mapsfrom (i,w)$. This works analogously for forests: let $F$ be a finite rooted $d$-ary forest (by which we will always mean a disjoint union of finitely many finite rooted complete subtrees of $\tree_d$) on $r$ roots with $n$ leaves. Each leaf of $F$ is an element $(i,u) \in \{1,\ldots,r\} \times X^*$ and we index the words in $X^*$ that occur in some leaf from left to right as $u_1,\ldots,u_n$. For each $1\le i\le r$ we choose indices $a_i, z_i$ such that the leaves of the $i$th tree are those $(i,u_j)$ with $a_i \le j \le z_i$. The homeomorphism defined by $F$ is $\varphi_F \colon C_r \leftarrow C_n, (i,u_jw) \mapsfrom (j,w)$ where $i$ is the index satisfying $a_i \le j \le z_i$. We take $\calF_d$ to be the category whose morphisms are the $\varphi_F$ for $F$ a finite rooted $d$-ary forest.

Let $\calW$ be the groupoid of homeomorphisms given by the (left) action of the wreath product $S_k \wr G$ on $C_k$, where $S_k$ acts on $\{1,\ldots,k\}$ and $G$ acts on the copies of $C$. In particular, note that $\calW$ contains no morphisms between distinct objects. Now define the category $\calN$, which is the one we are mainly interested in, to have as its morphisms all products of morphisms from $\calF_d$ and $\calW$.

\begin{lemma}\label{lem:decomposition}
Every morphism in $\calN$ can be written uniquely as $fw$ for some $f \in \calF_d$, $w \in \calW$.
\end{lemma}

\begin{proof}
First we prove uniqueness. Suppose $fw = f'w' \in \calN(C_k,C_\ell)$. Then $f = f'(w'w^{-1})$, with $w'w^{-1}$ a homeomorphism of $C_\ell$. Since $f$ and $f'$ identify $C_k$ and $C_\ell$ in an order-preserving way, in fact $w'w^{-1}$ has to be the identity.

Now we prove existence. It suffices to show that if $w' \in \calW(C_r,C_r)$ and $f' \in \calF_d(C_r,C_n)$ then there exists $f \in \calF_d(C_r,C_n)$ and $w \in \calW(C_n,C_n)$ such that $w'f' = fw$. Since $\calW(C_r,C_r)=S_r \wr G$, it is enough to assume $w'\in S_r \cup G^r$. Let $F'$ be the forest with $r$ roots and $n$ leaves such that $f' = \varphi_{F'}$. Say the leaves of $F'$ are $((1,u_{a_1}),\ldots,(1,u_{z_1}),\ldots,(r,u_{a_r}),\ldots,(r,u_{z_r}))$.

If $w' \in S_r$ then $w'f' = fw$ with $f = \varphi_F$ such that the leaves of $F$ are
\[
((1,u_{a_{(w')^{-1}(1)}}), \ldots, (1,u_{z_{(w')^{-1}(1)}}),\ldots,(r,u_{a_{(w')^{-1}(r)}}), \ldots, (r,u_{z_{(w')^{-1}(r)}}))
\]
and $w$ the permutation given by $w(j) = w'(a_i) + (j-a_i)$ where $i$ is such that $a_i \le j \le z_i$.

If $w' = (g_1',\ldots,g_r') \in G^r$ then $w'f' = fw$ with $f = f'$ and
\[
w = \sigma (g_1, \ldots, g_n)\text{,}
\]
where $g_j$ is the state of $g_i$ at $u_{j}$ for $a_i \le j \le z_i$, and $\sigma\in S_n$ is defined by $\sigma(j)=\theta_i(j-(a_i-1))+(a_i-1)$ whenever $a_i\le j\le z_i$, with $\theta_i\in S_{z_i-(a_i-1)}$ the permutation induced on $\{j-(a_i-1) \mid a_i\leq j \leq z_i\}$ by $g_i'$.
\end{proof}

A category is \emph{right-Ore} if it is \emph{cancellative}, meaning that $abc = ab'c$ implies $b=b'$, and has \emph{common right multiples}, meaning that if $a$ and $a'$ have the same target then there exist $b$ and $b'$ such that $ab = a'b'$.

\begin{lemma}
\label{lem:ore_noetherian}
The category $\calN$ is right-Ore.
\end{lemma}

\begin{proof}
That $\calN$ is cancellative is clear because all the morphisms are invertible when viewed as maps. Since morphisms in $\calW$ are invertible, to see that $\calN$ has common right multiples it suffices to see that $\calF_d$ has common right multiples. If $F$ and $F'$ are $d$-ary forests on $r$ roots, regarded as rooted subforests of the infinite rooted $d$-ary forest, and $F \cup F'$ is their union, then $\varphi_{F \cup F'}$ is a common right multiple of $F$ and $F'$.
\end{proof}

The \emph{fundamental group} $\pi_1(\calN,C_1)$ is by definition the group of all maps $C_1 \leftarrow C_1$ that arise as products of morphisms of $\calN$ and their inverses. That $\calN$ is right-Ore means that every such morphism can be written as $fg^{-1}$ with $f,g \in \calN(C_1,-)$.

\begin{observation}
The fundamental group of $\calN$ at $C_1$ is the R\"over--Nekrashevych group: $V_d(G) = \pi_1(\calN,C_1)$.
\end{observation}

\begin{proof}
By Lemmas~\ref{lem:decomposition} and~\ref{lem:ore_noetherian} every element of $\pi_1(\calN,C_1)$ can be written as
\[
(\varphi_{T_-}w_-)(\varphi_{T_+}w_+)^{-1} = \varphi_{T_-}w_-w_+^{-1}\varphi_{T_+}^{-1}
\]
where $T_-$ and $T_+$ are finite rooted complete $d$-ary trees with the same number of leaves, say $n$, and $w_-,w_+ \in S_n \wr G$. Writing $w_-w_+^{-1} = \sigma (g_1,\ldots,g_n)$ with $\sigma \in S_n$ and $(g_1,\ldots,g_n)\in G^n$ it is clear that the original homeomorphism is $[T_-,\sigma(g_1,\ldots,g_n),T_+]$ as described in Section~\ref{sec:nekr}. For the converse, it is clear that every homeomorphism in $V_d(G)$ arises this way.
\end{proof}

Note that every morphism in $\calW$ is invertible, whereas in $\calF_d$ only the identity morphisms are invertible. Hence a product $fw \in \calN(C_k,C_\ell)$ decomposed as in Lemma~\ref{lem:decomposition} is invertible if and only if $f$ is the identity if and only if $k = \ell$. In particular, $\calN^\times(C_k,C_k) = \calW(C_k,C_k)$. In the language of \cite{witzel17} this can be formulated as:

\begin{observation}
The map $\delta \colon \Ob(\calN) \to \N$, $C_k \mapsto k$ is a height function. In particular, $\calN$ is strongly Noetherian.
\end{observation}

The last ingredient to Theorem~\ref{thm:ore_criterion} is a Garside family. We need some more definitions to explain it. If a morphism decomposes as $f = abc$ then $a$ is called a \emph{left-factor} and $b$ is called a \emph{factor} of $f$. Let $\calE \subseteq \calN$ be a family of morphisms that contains all invertible morphisms and is closed under factors, i.e.\ if $f \in \calE$ then so is every factor of $f$. Assume that every morphism in $\calN$ can be written as a product of morphisms in $\calE$. Then $\calE$ is \emph{left-Garside} if every morphism $f \in \calN$ admits a greatest left-factor $\head(f) \in \calE$ in the following sense: $\head(f)$ is a left-factor of $f$ and if $s \in \calE$ is a left-factor of $f$ then it is a left-factor of $\head(f)$. We take $\calE$ to consist of morphisms $\varphi_Fw$ where $w \in \calW$ and $F$ is a rooted $d$-ary forest in which every tree is either trivial or a single $d$-caret.

\begin{lemma}
\label{lem:garside}
The family $\calE$ is a left-Garside family in $\calN$. It is closed under taking factors and is locally finite in the sense that up to right multiplication by invertible morphisms the set $\calE(C_k,-)$ is finite for every $k$.
\end{lemma}

\begin{proof}
That every morphism in $\calN$ can be written as a product of morphisms in $\calE$ follows from the fact that every $d$-ary forest can be built out of $d$-carets.
We need to check that every morphism in $\calN$ has a greatest left-factor in $\calE$. Since Lemma~\ref{lem:decomposition} says that up to right multiplication by an invertible morphism every morphism of $\calN$ lies in $\calF_d$, it suffices to check that $\calE$ restricts to a left-Garside family in $\calF_d$. Let $\varphi_F \in \calF_d(C_r,C_n)$. We claim that $\head(\varphi_F) = \varphi_{F'}$ where $F'$ is described as follows: (a) if the $i$th tree of $F$ is trivial then so is the $i$th tree of $F'$; (b) if the $i$th tree of $F$ is non-trivial then the $i$th tree of $F'$ is a single $d$-caret. We see that every left-factor of $\varphi_F$ needs to satisfy (a) and in order to be an element of $\calE$ the remaining trees cannot be more than single $d$-carets, showing that $\varphi_{F'}$ is as desired.

For local finiteness, by Lemma~\ref{lem:decomposition} every element of $\calE(C_k,-)$ is, up to right multiplication by an invertible morphism, one of the $2^k$ elements in $(\calE \cap \calF_d)(C_k,-)$.
\end{proof}

In order to apply Theorem~\ref{thm:ore_criterion} we need to check condition (\textsc{lk}), namely that for large enough $k$ the realization of a certain poset $E(C_k)$ is arbitrarily highly connected. The poset $E(C_k)$ is described as follows (see the paragraph before Theorem~3.12 in \cite{witzel17}). Its elements are equivalence classes of morphisms $h \in \calE(C_m,C_k)$ with $m < k$, where the equivalence relation is given by left multiplication by a morphism in $\calN^\times(C_m,C_m) = \calW(C_m,C_m)$. The order on equivalence classes is given by $[h_1] \le [h_2]$ if $h_2 = h'h_1$ for some $h' \in \calE$. Rather than directly dealing with the realization $\abs{E(C_k)}$ we will later see that $E(C_k)$ is the face poset of a simplicial complex $X_k$ described as follows: the vertices of $X_k$ are equivalence classes of morphisms $h \in \calE(C_{k-(d-1)},C_k)$ up to left multiplication by a morphism in $\calN^\times(C_{k-(d-1)},C_{k-(d-1)})$. Vertices $[h_1],\ldots,[h_n]$ form a simplex if there is a morphism $h \in \calE(C_{k-n(d-1)},C_k)$ that has each of them as a right factor.

Let $h = \varphi_{F}\sigma(g_1,\ldots,g_k) \in \calE(C_{k-(d-1)},C_k)$ be a representative of a vertex of $X_k$, so $F$ has a single non-trivial tree, $\sigma \in S_k$ and $(g_1,\ldots,g_k) \in G^k$. Note that if the index of the non-trivial tree in $F$ is $i$ then $\varphi_{F}(\{i\} \times C) = \{i,\ldots,i+d-1\} \times C$. We call the \emph{support} of $h$ the set $\{\sigma^{-1}(i), \ldots, \sigma^{-1}(i+d-1)\}$.

\begin{lemma}
The support is invariant under left multiplication by an invertible morphism.
\end{lemma}

\begin{proof}
It is clear that left multiplying by an element of $G^{k-(d-1)}$ does not change the support. Let $\varphi_F$ be as above and let $\tau \in S_{k-(d-1)}$. Then $\tau \varphi_F = \varphi_{F'} \tau'$ where $\tau'$ takes $\{i,\ldots,i+d-1\}$ to $\{\tau(i),\ldots,\tau(i)+d-1\}$ and the non-trivial tree of $F'$ has index $\tau(i)$. The claim follows.
\end{proof}

\begin{lemma}\label{lem:disjoint_support}
Vertices $[h_1], \ldots,[h_n]$ form a simplex in $X_k$ if and only if they have disjoint support. In this case the morphism $h \in \calE(C_{k-n(d-1)},C_k)$ having each $h_i$ as a right factor is unique up to left multiplication by an invertible morphism.
\end{lemma}

\begin{proof}
That vertices forming a simplex have disjoint support is clear. Conversely, let $h_i =\varphi_{F_i}\sigma_i \vec{g}_i \in \calE(C_{k-(d-1)},C_k)$. Left multiplying by a suitable element of $S_{k-(d-1)}$, for each $i$ we can arrange the non-trivial tree of $F_i$ to be the $i$th, and moreover we can alter $h_i$ by arbitrary elements of $G$ outside its support. Thus we construct the morphism $h$ as follows: we define $\vec{g}$ to coincide with $\vec{g}_i$ on the support of $h_i$ and to be arbitrary everywhere else; we define $\sigma$ to take the support of $h_i$ to $\{i,\ldots,i+d-1\}$, to act on it as $\sigma_i$, and to be arbitrary everywhere else; we define $F$ to be the forest whose $i$th tree is the same as that of $F_i$ and which has all but the first $n$ trees trivial. Then the morphism $h = \varphi_F\sigma \vec{g}$ has all the $h_i$ as right factors. The remaining choices concern how $S_k \wr G$ acts on $\{nd+1,\ldots,k\} \times C$, and they are reflected as left multiplication by an element of $\calW(C_{k-n(d-1)},C_{k-n(d-1)})$ fixing $\{1,\ldots,n\} \times C$.
\end{proof}

\begin{corollary}\label{cor:unsubdivision}
The poset $E(C_k)$ is isomorphic to the face poset of $X_k$. Thus $\abs{E(C_k)}$ is isomorphic to the barycentric subdivision of $X_k$, and in particular they are homeomorphic.
\end{corollary}

\begin{proof}
Note that any morphism $h$ in $\calE(C_{k-n(d-1)},C_k)$ is of the form $\varphi_Fw$ where $F$ has $n$ non-trivial trees and $w\in \calW(C_k,C_k)$. Hence $h$ defines an $(n-1)$-simplex in $X_k$ consisting of equivalence classes of morphisms $\varphi_{F_i}w \in \calE(C_{k-(d-1)},C_k)$, $1\le i \le n$, where $F_i$ is such that the $i$th non-trivial tree of $F$ coincides with the tree of $F_i$ on the same root and all other trees of $F_i$ are trivial. The map that sends $[h]$ to that simplex is a poset map $E(C_k) \to X_k$ that is surjective by definition of $X_k$. It is injective by the uniqueness statement of Lemma~\ref{lem:disjoint_support}.
\end{proof}

\begin{corollary}\label{cor:flag}
The complex $X_k$ is flag.
\end{corollary}

\begin{proof}
We need to show that vertices that can be pairwise joined by edges form a simplex. Such vertices have pairwise disjoint support, so the claim follows from the existence statement of Lemma~\ref{lem:disjoint_support}.
\end{proof}

To verify (\textsc{lk}) we use a result due to Belk and Forrest \cite{belk15}. A simplex $\sigma$ in a simplicial flag complex is called a \emph{$k$-ground}, for $k \in \N$, if every vertex of $X$ is adjacent to all but at most $k$ vertices of $\sigma$. The complex is said to be \emph{$(n,k)$-grounded} if there is an $n$-simplex that is a $k$-ground.

\begin{theorem}[{\cite[Theorem~4.9]{belk15}}]\label{thm:belk_forrest}
For $m,k \in \N$ every $(mk,k)$-grounded flag complex is $(m-1)$-connected.
\end{theorem}

\begin{remark}
The reference makes additional assumptions that are not necessary: the assumption that $m,k \ge 1$ can be removed by observing that a $(0,k)$-grounded complex is always non-empty and that a $(0,0)$-grounded complex is contractible. The assumption that the complex be finite can be removed using that in an arbitrary complex, any sphere is supported on a finite subcomplex.
\end{remark}

\begin{lemma}\label{lem:ground}
The complex $X_k$ is $(\floor{k/d}-1,d)$-grounded.
\end{lemma}

\begin{proof}
Let $h \in \calE(C_{k - \floor{k/d}(d-1)},C_k)$ be arbitrary. We claim that the $(\floor{k/d}-1)$-simplex $[h]$ is a $d$-ground. Indeed, if $v \in \calE(C_{k-(d-1)},C_k)$ represents a vertex then its support has cardinality $d$, so its support can be non-disjoint from the support of at most $d$ of the vertices of $[h]$ (which have pairwise disjoint supports). The result is now clear from Lemma~\ref{lem:disjoint_support}.
\end{proof}

\begin{corollary}\label{cor:connectivity}
The complex $X_k$ is $(\floor{\frac{k-d}{d^2}}-1)$-connected.
\end{corollary}

\begin{proof}
This follows by combining Theorem~\ref{thm:belk_forrest} with Corollary~\ref{cor:flag} and Lemma~\ref{lem:ground}, and the calculation
\[
\dfloor{\frac{\floor{\frac{k}{d}}-1}{d}} = \dfloor{\frac{\floor{\frac{k}{d}-1}}{d}} = \dfloor{\frac{\frac{k}{d}-1}{d}} = \dfloor{\frac{k - d}{d^2}}\text{.}\qedhere
\]
\end{proof}

\begin{theorem}\label{thm:positive}
Let $G \le \Aut(\tree_d)$ be self-similar. If $G$ is of type $\F_{n-1}$ then so is $V_d(G)$.
\end{theorem}

\begin{proof}
We want to apply Theorem~\ref{thm:ore_criterion} with $\calC = \calN$, $* = C_1$, $\calE = \calE$ and $\delta \colon C_k \mapsto k$. The assumptions in the running text are verified in Lemma~\ref{lem:ore_noetherian} and Lemma~\ref{lem:garside}. Condition (\textsc{lk}) follows from Corollary~\ref{cor:unsubdivision} and Corollary~\ref{cor:connectivity}. The group $\calN^\times(C_k,C_k)$ is $G \wr S_k$, which is virtually isomorphic to $G^k$ and therefore of type $\F_n$ by assumption. Applying the theorem it follows that $V_d(G) = \pi_1(\calN,C_1)$ is of type $\F_n$.
\end{proof}

\section{Disproving type $\F_n$}\label{sec:negative}

In this section we show that any group $G$ with a faithful, self-similar finite-state action on $\tree_{d-1}$ admits an action with these same properties on $\tree_d$ such that if $G$ is not of type $\F_n$ then neither is $V_d(G)$. This will provide us with many examples of R\"over--Nekrashevych groups that are not of type $\F_n$. This is the most important part of the article for various reasons. One reason is that most generalized Thompson groups are of type $\F_\infty$, so the property that we are aiming for is simply not satisfied by them. Another reason is that even if a generalized Thompson group is not of type $\F_n$, there is no established way to prove this. Often, when a group is of type $\F_{n-1}$ but not of type $\F_n$ it has a natural action on an $n$-dimensional space that can be utilized, but no such action is available to us. In the few known examples of generalized Thompson groups that are not of type $\F_\infty$, individually tailored arguments have been used; see for example \cite[Section~8.2]{witzel14} and \cite{witzel16}. Here we employ a more robust strategy using quasi-retracts.

The following property is the key to our approach.

\begin{definition}[Persistent]
 Let $G\le \Aut(\tree_d)$ be self-similar. We say $g\in G$ is \emph{$i$-persistent} if, in the wreath recursion $g=\rho(g)(g_1,\dots,g_d)$ we have $g_i=g$. We call $G$ \emph{$i$-persistent} if every element of $G$ is $i$-persistent. We say $g\in G$ is \emph{persistent} if it is $i$-persistent for some $i$, and that $G$ is \emph{persistent} if there exists $i$ such that every $g\in G$ is $i$-persistent. We will also refer to a group $G$ as having a \emph{faithful ($i$-)persistent action} on $\tree_d$ if $G$ acts faithfully on $\tree_d$ and its resulting image in $\Aut(\tree_d)$ is (self-similar and) ($i$-)persistent.
\end{definition}

\begin{remark}
 One can view the property of being persistent as a strong negation of the property of being contracting. A self-similar group is \emph{contracting} if it admits a finite subset $S$ (with the smallest such subset called its \emph{nucleus}) such that for every element $g$ of the group there exists an $N$ such that every state of $g$ below the $N$th level lies in $S$. If the group is persistent then for every element $g$ and every level, at least one of the states of $g$ at that level is $g$ itself, and thus unless the group is finite it can have no such $S$.
\end{remark}

Note that up to isomorphism, $G$ being $i$-persistent and $j$-persistent are equivalent properties for any $1\le i,j\le d$. In particular if $G$ is persistent then up to isomorphism it is $d$-persistent.

\begin{lemma}[Creating persistence]\label{lem:persist}
 Let $d\ge 3$. Let $G$ be a group with a faithful self-similar action on $\tree_{d-1}$. Then $G$ admits a faithful, persistent self-similar action on $\tree_d$. If the action on $\tree_{d-1}$ is finite-state then so is the action on $\tree_d$. If the action on $\tree_{d-1}$ is coarsely diagonal then so is the action on $\tree_d$. Moreover, an $i$-persistent such action can be found for any $1\le i\le d$.
\end{lemma}

\begin{proof}
 Say the wreath recursion for the action on $\tree_{d-1}$ is given by $g=\rho(g)(g_1,\dots,g_{d-1})$. Now define an action of $G$ on $\tree_d$ via the wreath recursion $g=\iota\circ\rho(g)(g_1,\dots,g_{d-1},g)$, where $\iota\colon S_{d-1}\to S_d$ is induced by the inclusion $\{1,\dots,d-1\}\to\{1,\dots,d\}$. This is self-similar and $d$-persistent by construction. If the original action is finite-state then so is the new action. If the original action is coarsely diagonal then so is the new action. It remains to check the new action is faithful. Indeed, the inclusion $\{1,\dots,d-1\}\to\{1,\dots,d\}$ also defines an embedding $\tree_{d-1}\to\tree_d$, with image invariant under the action of $G$ on $\tree_d$. The action of $G$ restricted to this image is the same as the action of $G$ on $\tree_{d-1}$, which is faithful, so the action on $\tree_d$ is faithful as well. If we want an $i$-persistent action for some $i$ other than $d$, we need only conjugate by the tree automorphism induced by the transposition of $\{1,\dots,d\}$ switching $i$ and $d$, which preserves the properties of being self-similar, finite-state, and coarsely diagonal.
\end{proof}

For a $d$-persistent group $G$ we have the following result, which allows us to extract an element of $G$ out of an element of $V_d(G)$ in a canonical way.

\begin{lemma}\label{lem:clone_stable}
 Let $G\le \Aut(\tree_d)$ be a $d$-persistent self-similar group. Whenever two triples $(T_-,\sigma(g_1,\dots,g_n),T_+)$ and $(T_-',\sigma'(g_1',\dots,g_{n'}'),T_+')$ represent the same element of $V_d(G)$, we have $g_n=g_{n'}'$.
\end{lemma}

\begin{proof}
 By Lemma~\ref{lem:iff_exp} it suffices to prove this in the case when $(T_-',\sigma'(g_1',\dots,g_{n'}'),T_+')$ is an expansion of $(T_-,\sigma(g_1,\dots,g_n),T_+)$, and by induction it suffices to assume $T_+'$ is $T_+$ with a single $d$-caret added to one of its leaves, say the $k$th. Now Lemma~\ref{lem:expand_one} implies that $(g_1',\dots,g_{n'}')=(g_1,\dots,g_{k-1},g_k^1,\dots,g_k^d,g_{k+1},\dots,g_n)$, where $g_k=\rho(g_k)(g_k^1,\dots,g_k^d)$ is the wreath recursion of $g_k$. In particular if $k<n$ then $g_{n'}'=g_n$ trivially, and if $k=n$ then $g_{n'}'=g_n^d$, which equals $g_n$ since $G$ is $d$-persistent.
\end{proof}

\begin{proposition}\label{prop:quasi_retract}
 Let $G\le \Aut(\tree_d)$ be a finitely generated, persistent, finite-state self-similar group. Then there exists a quasi-retraction $V_d(G) \to G$.
\end{proposition}

\begin{proof}
 We can assume that $G$ is $d$-persistent. Let $\iota = \iota_\varnothing\colon G\to V_d(G)$ be the monomorphism $g\mapsto [1_1,g,1_1]$. Let $r\colon V_d(G) \to G$ be the function $[T_-,\sigma(g_1,\dots,g_n),T_+]\mapsto g_n$. Note that $r$ is well defined by Lemma~\ref{lem:clone_stable}. Clearly $r\circ\iota=\id_G$. Since $\iota$ is a homomorphism, it is $(C,0)$-Lipschitz for some $C\ge 1$. It remains to prove that $r$ is $(C',D)$-Lipschitz for some $C'\ge 1$ and $D\ge 0$, and in fact we will prove it is $(1,0)$-Lipschitz given the right choice of generating set.

Let $S_G$ be a finite, self-similar, symmetric generating set for $G$, which exists since $G$ is finitely generated and finite-state. Let $S_{V_d}$ be a finite symmetric generating set for $V_d$. By Lemma~\ref{lem:nekr_gens}
\[
S_{V_d(G)}\defeq \iota_1(S_G) \cup S_{V_d}
\]
is a (finite, symmetric) generating set of $V_d(G)$, where $\iota_1$ is as in Lemma~\ref{lem:nekr_gens}. We will consider the word metric on $G$ using $S_G$, and the word metric on $V_d(G)$ using $S_{V_d(G)}$. To show that $r$ is $(1,0)$-Lipschitz we need to show that if two elements of $V_d(G)$ are adjacent in the Cayley graph their images under $r$ are adjacent or coincide. For technical reasons we consider left Cayley graphs.

Let $x=[T_-,\sigma(g_1,\dots,g_n),T_+]\in V_d(G)$ and $s=[U_-,\tau(h_1,\dots,h_m),U_+]\in S_{V_d(G)}$ be arbitrary. Up to possibly expanding, we can assume $U_+=T_-$ (so $m=n$) and $h_i\in S_G \cup \{\id\}$ for all $i$, since $S_G$ is a self-similar set. Now
\[
sx=[U_-,\tau(h_1,\dots,h_n)\sigma(g_1,\dots,g_n),T_+] = [U_-,\tau\sigma(h_{\sigma(1)}g_1,\dots,h_{\sigma(n)}g_n),T_+]\text{,}
\]
so $r(sx)=h_{\sigma(n)}g_n$. Since $h_{\sigma(n)}\in S_G \cup \{\id\}$ and $r(x)=g_n$, this implies that $r(sx)$ and $r(x)$ are adjacent or coincide.
\end{proof}

Note in the proof of Proposition~\ref{prop:quasi_retract} that $h_{\sigma(n)}$ is not necessarily $r(s)$, unless for example $\sigma(n)=n$. In particular we explicitly see the failure of $r$ to be a homomorphism in the equation $r(sx)=h_{\sigma(n)}r(x)$. Since $V_d(G)$ may very well be virtually simple, we should not expect $r$ to be a homomorphism, and indeed it need not be.

Now that we have a quasi-retract from $V_d(G)$ to $G$, Theorem~\ref{thm:alonso} says that type $\F_n$ is passed from $V_d(G)$ to $G$ for $n \ge 2$. In order to also include $n = 1$ we make the following addition to Theorem~\ref{thm:alonso}. Here $d_S$ means the word metric with respect to $S$.

\begin{lemma}\label{lem:alonso_fg}
Let $Q$ and $H$ be countable groups and let $(S_i)_{i \in \N}$, $S_i\subseteq Q$ and $(T_i)_{t \in \N}$, $T_i\subseteq H$ be ascending sequences of finite sets such that $Q$ is generated by $\bigcup_i S_i$ and $H$ is generated by $\bigcup_i T_i$. Assume that there is a map $r \colon H \to Q$ such that for every $i$
\begin{enumerate}
\item if $d_{T_i}(x,y) < \infty$ then $d_{S_i}(r(x),r(y)) < \infty$, and
\item every element of $Q$ has finite distance with respect to $d_{S_i}$ to $r(H)$.
\end{enumerate}
If $H$ is finitely generated then so is $Q$.
\end{lemma}

\begin{proof}
If $H$ is finitely generated then some $T_i$ generates it, so any two points in $H$ have finite distance with respect to $d_{T_i}$. By the first assumption it follows that any two points in $r(H)$ have finite distance with respect to $d_{S_i}$. The second assumption then ensures that any two points in $Q$ have finite distance with respect to $d_{S_i}$, which means that $S_i$ generates $Q$.
\end{proof}

\begin{lemma}\label{lem:infinite_generation}
Let $G\le \Aut(\tree_d)$ be a persistent, finite-state self-similar group. If $V_d(G)$ is finitely generated then so is $G$.
\end{lemma}

\begin{proof}
The proof is based on that of Proposition~\ref{prop:quasi_retract}. Note that $G$ is countable since it lies in the finitely generated group $V_d(G)$. Take $S_i$ to be an ascending, exhaustive sequence of finite, self-similar, symmetric subsets of $\iota_1(G)$, and take $T_i = S_i \cup S_{V_d}$ for each $i$. Define $r \colon V_d(G) \to G$ as before, so for every $i$ the map $r \colon (V_d(G),d_{T_i}) \to (G,d_{S_i})$ is surjective and $(1,0)$-Lipschitz. This shows that the hypotheses of Lemma~\ref{lem:alonso_fg} are satisfied and the claim follows.
\end{proof}

Combining Theorem~\ref{thm:alonso} with Proposition~\ref{prop:quasi_retract} and Lemma~\ref{lem:infinite_generation} we get:

\begin{theorem}\label{thm:negative}
 Let $G\le \Aut(\tree_d)$ be a persistent, finite-state self-similar group. If $V_d(G)$ is of type $\F_n$ (or $\FP_n$) then so is $G$.
\end{theorem}

\begin{remark}\label{rmk:retract}
 If $Q$ is a retract of a group $H$ then $Q$ has all the finiteness properties that $H$ has. This has a homological proof using the Bieri--Eckmann Criterion, as in \cite[Proposition~4.1]{bux04}, but it also has a geometric proof given by applying Theorem~\ref{thm:alonso}. The relationship between the non-$\FP_n$ proofs in \cite{witzel14} and in the present article is similar: in \cite{witzel14} a sophisticated variant of the homological argument is used, whereas here we instead use a geometric argument, by putting ourselves into a position where we can apply Theorem~\ref{thm:alonso}.
\end{remark}

\section{Self-similar groups separated by finiteness properties}
\label{sec:self_similar_fin_props}

The remaining ingredient needed to prove the Main Theorem is to find a family of coarsely diagonal, finite-state self-similar groups separated by finiteness properties. Self-similar groups separated by finiteness properties were first exhibited by Bartholdi, Neuhauser, and Woess in \cite{bartholdi08}, followed by Kochloukova and Sidki in \cite{kochloukova17} using virtual endomorphisms. Inspired by their approach we consider certain metabelian $S$-arithmetic groups separated by finiteness properties that naturally act on Bruhat--Tits trees and show that these natural actions are coarsely diagonal, finite-state and self-similar.

We briefly recall some facts about rings of $S$-integers in positive characteristic and establish notation along the way. We refer to \cite{weil67,artin67,niederreiter09} for more details and to \cite[Section~1.5]{witzel14phd} for an exposition using the same notation. Although our description is fairly general, the familiar case of a rational function field $\FF_q(t)$ is sufficient for the purpose of the Main Theorem and we will indicate the special cases throughout.

A \emph{global function field} $k$ is a finite extension of the field of rational functions $\FF_q(t)$ over a finite field $\FF_q$. A \emph{place} $[\nu]$ is an equivalence class of a discrete valuation $\nu \colon k \to \R \cup \{\infty\}$ modulo scaling ($\nu(\alpha) = \infty$ if and only if $\alpha = 0$). We say that an element $\alpha \in k$ has a \emph{pole} (respectively a \emph{zero}) at $[\nu]$ if $\nu(\alpha) < 0$ (respectively $\nu(\alpha) > 0$). If $S$ is a non-empty finite set of places, the \emph{ring of $S$-integers} $\calO_S$ consists of those elements of $k$ that have no poles at places outside $S$.

\begin{example}
The places of $\FF_q(t)$ correspond to points of the projective line over $\FF_q$. More precisely, there is a valuation $\nu_\alpha$ for every irreducible polynomial $\alpha \in \FF_q[t]$ of degree $e$, given by $\nu(\alpha^f (\beta/\gamma)) = f$ for $\alpha \nmid \beta, \gamma$, and corresponding to the point at infinity there is a valuation $\nu_\infty(\beta/\gamma) = \deg(\gamma) - \deg(\beta)$. These represent all places of $\FF_q(t)$. If $S = \{[\nu_\infty], [\nu_{\alpha_1}], \ldots, [\nu_{\alpha_n}]\}$ is a set of places containing $[\nu_\infty]$ then $\calO_S = \FF_q[t,\alpha_1^{-1},\ldots,\alpha_n^{-1}]$.
\end{example}

The completion $K = k_s$ of $k$ at a place $s$ is a \emph{local field}. In particular, its ring of integers $\calO$ has a unique maximal ideal $\frakm$ and the residue field $\kappa = \calO/\frakm = \FF_{q^e}$ is finite. Here $e$ is the \emph{degree} of $s$, which is the degree of $\kappa \cong \FF_{q^e}$ over the field of constants $\FF_q < k$. Up to scaling $K$ has a unique valuation $\nu(\alpha) = m$ with $m = \sup\{\ell \mid \alpha \in \frakm^\ell\}$ and the restriction to $k$ represents $s$. A \emph{uniformizing element} $\pi$ is a generator of $\frakm$ in $\calO$. If $\pi$ is a uniformizing element in a local field $K$ of positive characteristic then $K$ is the field of Laurent series $\FF_{q^e}\lseries{\pi}$. Its ring of integers is the ring of power series $\calO = \FF_{q^e}\pseries{\pi}$ and the maximal ideal is $\frakm = \pi\FF_{q^e}\pseries{\pi}$.

\begin{example}
The completion of $\FF_q(t)$ with respect to $[\nu_\alpha]$ is $\FF_{q^e}\lseries{\alpha^{-1}}$ if $\deg \alpha = e$. The completion of $\FF_q(t)$ with respect to $[\nu_\infty]$ is $\FF_q\lseries{t}$.
\end{example}

Now we discuss the relevant groups. Let $K$ be a local field of positive characteristic $p$ and let $\nu$, $\calO$, $\frakm$, $\kappa$ and $q$ be as above. Let $G_a$ and $G_m$ be the additive and multiplicative algebraic groups respectively, and consider the algebraic group $\AGL_1 = G_m \ltimes G_a$ (with the usual action) which can be thought of as the group of matrices given by
\[
\AGL_1(R) = \left\{\left. \begin{pmatrix} \alpha & \beta\\  & 1 \end{pmatrix} \right| \alpha \in R^\times, \beta \in R\right\}\text{.}
\]
This group is a subgroup of $\PGL_2$ and therefore $\AGL_1(K)$ acts on the Bruhat--Tits tree $\tree_K$ associated to $\PGL_2(K)$. Since we only want to describe the action of $\AGL_1(K)$, this tree can be realized particularly easily. The vertices of $\tree_K$ are residue classes $\gamma \mod \frakm^e$ with $\gamma \in K$ and $e \in \Z$, and two residue classes $\gamma \mod \frakm^e$ and $\delta \mod \frakm^f$ are joined by an edge whenever $\abs{e-f} = 1$ and $\gamma \equiv \delta \mod \frakm^{\min\{e,f\}}$. Here and in what follows we write elements of $\AGL_1(K)$ as pairs $(\alpha, \beta)$ with $\alpha \in G_m(K) = K^\times$ and $\beta \in G_a(K) = K$. An element $(\alpha,\beta) \in \AGL_1(K)$ acts via
\[
(\alpha,\beta).(\gamma \mod \frakm^e) = \alpha\gamma + \beta \mod \frakm^{e + \nu(\alpha)}\text{.}
\]

The boundary of $\tree_K$ consists of ends corresponding to elements $\gamma$ of $K$, which are approached by the sequence $\gamma \mod \frakm^e$ for $e \to \infty$, and an end, denoted $\infty$, which is approached by any sequence $\gamma \mod \frakm^e$ for $e \to -\infty$. In fact, $\partial \tree_K$ is homeomorphic to the projective line over $K$ in an $\AGL_1(K)$-equivariant way. In particular $\AGL_1(K)$ fixes the end $\infty$.

Moreover, the action of $\AGL_1(\calO)$ fixes the horoball around $\infty$ whose vertices are the $\gamma \mod \frakm^e$ with $e \le 0$. Consequently it acts on each of the rooted trees consisting of vertices $\gamma \mod \frakm^e$ with $e \ge 0$ and $\gamma \equiv \delta \mod \calO$, for any choice of $\delta \in K$. We denote by $\tree_\calO$ the tree corresponding to $\delta = 0$. In particular the vertices of $\tree_\calO$ are cosets $\gamma \mod \frakm^e$ with $\gamma \in \calO$ and $e \ge 0$. The root of $\tree_\calO$ is $\gamma \mod \frakm^0$ for any $\gamma$. We fix an order on $\kappa$ and obtain an induced order on the vertices of $\tree_\calO$, which in particular identifies $\tree_\calO$ with $\tree_{\abs{\kappa}}$.

We introduce the following notation: if $\gamma = \sum c_i \pi^i \in K$ is a Laurent series then $[x]_a^b$ is the truncated series $\sum_{a \le i < b} c_i \pi^i$. If the sub- or superscript is omitted, then so is the corresponding condition.

\begin{lemma}\label{lem:arith_states}
Let $\alpha \in G_m(\calO)$ and $\beta \in G_a(\calO)$, and let $\gamma \mod \frakm^e$ be a vertex of $\tree_\calO$.
\begin{enumerate}
\item The state of $\beta$ at $\gamma \mod \frakm^e$ is $\pi^{-e}[\beta]_e \in G_a(\calO)$.
\item The state of $\alpha$ at $\gamma \mod \frakm^e$ is $(\alpha,\delta) \in \AGL_1(\calO)$ where $\delta = \pi^{-e}\left(\alpha[\gamma]^e - [\alpha\gamma]^e\right)$.
\end{enumerate}
\end{lemma}

\begin{proof}
The ends of the subtree with root $\gamma \mod \frakm^e$ are of the form $[\gamma]^e + \pi^e\zeta$ for $\zeta \in \calO$. The element $\beta$ takes such an end to
\[
\beta + [\gamma]^e + \pi^e\zeta = [\beta+\gamma]^e + \pi^e(\pi^{-e}[\beta]_e + \zeta) \text{,}
\]
showing the first claim.

Similarly, $\alpha$ takes $[\gamma]^e + \pi^e\zeta$ to
\[
\alpha \cdot \left([\gamma]^e + \pi^e\zeta\right) = \alpha[\gamma]^e + \pi^e\alpha\zeta = [\alpha\gamma]^e + \pi^e\left(\alpha\zeta + \pi^{-e}\left(\alpha[\gamma]^e - [\alpha\gamma]^e\right)\right)\text{,}
\]
showing the second claim. Note that the first $(e-1)$ coefficients of $\alpha[\gamma]^e$ and $[\alpha\gamma]^e$ coincide so $\delta = \pi^{-e}\left(\alpha[\gamma]^e - [\alpha\gamma]^e\right)$ really does lie in $\calO$.
\end{proof}

In particular, Lemma~\ref{lem:arith_states} provides a criterion for $\AGL_1(R)$ to be self-similar for $R < \calO$: for every $e \ge  0$ and every $\gamma \in \calO$, whenever $\beta \in R$ then $\pi^{-e}[\beta]_e$ has to be in $R$ as well, and if $\beta$ is invertible then in addition $\pi^{-e}(\beta[\gamma]^e - [\beta\gamma]^e)$ has to be in $R$.

Now let $k$ be a global function field and let $s$ be a place of degree $1$. Let $K = k_s$ be the completion of $k$ at $s$ and, as before, let $\calO$ be its ring of integers. Let $\pi$ be a uniformizing element, which we take to lie in $k$. Using this setup we have $K = \FF_q\lseries{\pi}$ and $\calO = \FF_q\pseries{\pi}$.

In order to get finite-state actions we will have to restrict to rational function fields because of the following:

\begin{lemma}\label{lem:global}
A Laurent series $\alpha \in K$ lies in $\FF_q(t)$ if and only if its coefficients are eventually periodic.
\end{lemma}

\begin{proof}
It is well known that a Laurent series is rational if and only if it satisfies a linear recurrence rule. Since the coefficient field is finite, such a sequence is eventually periodic.

Conversely, if $\alpha$ has period $\ell$ starting with the $m$th coefficient then it is also periodic with period $n = \ell \ceil{m/\ell}$ starting with the $n$th coefficient. Thus there is an $f$ such that $\pi^f\alpha = \beta/(\pi^n-1)$ for a $\beta$ a polynomial in $\pi$ of degree at most $f + n-1$. It follows that $\alpha = \beta/(\pi^{n+f} - \pi^{f})$ is rational.
\end{proof}

\begin{lemma}
\label{lem:arith_individual_finite-state}
If $k  = \FF_q(t)$ is a rational function field then the action of $\AGL_1(k \cap \calO)$ on $\tree_\calO$ is finite-state.
\end{lemma}

\begin{proof}
Using Lemma~\ref{lem:product_of_finite_state} it suffices to show that $G_m(k \cap \calO)$ and $G_a(k \cap \calO)$ are finite-state.

If $\beta \in G_a(k \cap \calO)$ then it is eventually periodic by Lemma~\ref{lem:global}. Therefore the set of $\pi^{-e}[\beta]_e$ for $e \ge 0$, which is its set of states by Lemma~\ref{lem:arith_states}, is finite.

The state of $\alpha \in G_m(k \cap \calO)$ at $\gamma \mod \frakm^e$ is $(\alpha,\delta) \in \AGL_1(\calO)$ with $\delta = \pi^{-e}\left(\alpha[\gamma]^e - [\alpha\gamma]^e\right)$ by Lemma~\ref{lem:arith_states}. Note that $\delta = \pi^{-e}[\alpha[\gamma]^e]_e$, so $\delta$ is rational since $\alpha$ is. In particular every $\delta$ that occurs lies in $G_a(k\cap \calO)$, and hence is finite-state by the previous paragraph. It now suffices to show that only finitely many $\delta$ can occur for a fixed $\alpha$.

To see this, write $\alpha = \sum_i a_i \pi^i$, $\gamma = \sum_i b_i \pi^i$ and $\delta = \sum_i c_i \pi^i$. Then
\[
c_i = \sum_{\substack{0\le \ell \le e-1\\j+\ell = e + i}} a_j b_\ell = \sum_{\ell=0}^{e-1} a_{i + (e-\ell)} b_\ell\text{.}
\]
It follows that $\delta$ is eventually periodic with the same period as $\alpha$ and, moreover, the periodicity of $\delta$ starts (at the latest) at the same index as for $\alpha$. Since these conclusions hold regardless of $\gamma$, indeed there are only finitely many $\delta$ per $\alpha$.
\end{proof}

\begin{remark}
The condition on $k$ is clearly necessary: if $k$ is a global function field that is not a rational function field then it contains an aperiodic power series $\alpha$ by Lemma~\ref{lem:global}. Lemma~\ref{lem:arith_states} shows that $\alpha$ regarded as an element of $G_a(k)$ is not finite-state.
\end{remark}

We now return to $k$ being an arbitrary global function field. No restrictions are necessary to get coarse diagonality:

\begin{lemma}
\label{lem:arith_coarsely_diagonal}
The action of $\AGL_1(\calO)$ on $\tree_{\calO}$ is coarsely diagonal.
\end{lemma}

\begin{proof}
By Lemma~\ref{lem:arith_states}, every state of $g = (\alpha,\beta) \in \AGL_1(\calO)$ is of the form $g' = (\alpha,\delta)$ for some $\delta$. It follows that $(g')^{-1}g \in G_a(\calO)$, which has exponent $p$.
\end{proof}

From now on let $S$ be a finite, non-empty set of places not containing $s$ and let $\calO_S$ be the ring of $S$-integers of $k$. We choose the uniformizing element $\pi \in \frakm$ to lie in $\calO_S$, which is possible by the Riemann--Roch theorem \cite[Proposition~3.6.14]{niederreiter09} (for example if $[\nu_\infty] \in S$ then $k \cap \calO_{\{[\nu_\infty]\}} \ge \FF_q[t]$ and we can take $\pi = t$).

Our interest in the groups $\AGL_1(\calO_S)$ is based on the following result.

\begin{theorem}[{\cite{bux97,kochloukova96}}]\label{thm:metabelian_fin_props}
Let $k$ be a global function field and let $S$ be a non-empty finite set of places. The group $\AGL_1(\calO_S)$ is of type $\F_{\abs{S}-1}$ but not of type $\FP_{\abs{S}}$.
\end{theorem}

\begin{remark}
The group $\AGL_1(\calO_S)$ acts on the product of trees $\prod_{u \in S} \tree_{k_u}$ as a discrete group and Theorem~\ref{thm:metabelian_fin_props} is proved by describing a cocompact subspace for the action.
\end{remark}

We need one more general lemma, which is our main reason to restrict $s$ to degree $1$:

\begin{lemma}
\label{lem:pole}
For a series $\alpha \in \calO$ we have $\alpha \in \calO_S$ if and only if $\beta = \pi^{-j}[\alpha]_j \in \calO_S$.
\end{lemma}

\begin{proof}
Consider the relation
\[
\alpha = \pi^j\beta + \gamma
\]
where $\gamma = [\alpha]^j$ is a polynomial in $\pi \in \calO_S$. It follows that $\alpha \in \calO_{S \cup \{s\}}$ if and only if $\beta \in \calO_{S \cup \{s\}}$, and the explicit construction shows that $\beta$ does not have a pole at $s$.
\end{proof}

It follows that the groups $\AGL_1(\calO_S)$ suit our needs:

\begin{proposition}
\label{prop:arith_self_similar}
Let $k$ be a global function field and let $S$ be a finite non-empty set of places not containing $s$.
The action of $\AGL_1(\calO_S)$ on $\tree_{\calO}$ is self-similar and coarsely diagonal. If $k$ is a rational function field, the action is also finite-state.
\end{proposition}

\begin{proof}
We have $\calO_S \subseteq \calO$ by the assumption that $s \not\in S$. It therefore follows from Lemma~\ref{lem:arith_individual_finite-state} and Lemma~\ref{lem:arith_coarsely_diagonal} that the action is coarsely diagonal and that it is finite-state if $k$ is a rational function field. We need to check self-similarity. It suffices to show that states of elements of $G_a(\calO_S) \cup G_m(\calO_S)$ lie in $\AGL_1(\calO_S)$.

The states of $\beta \in G_a(\calO_S)$ are of the form $\pi^{-e}[\beta]_e$ by Lemma~\ref{lem:arith_states}, which lie in $G_a(\calO_S)$ by Lemma~\ref{lem:pole}.

For $\alpha \in G_m(\calO_S)$, by Lemma~\ref{lem:arith_states}, we have to verify that $\delta = \pi^{-e}(\alpha[\gamma]^e - [\alpha\gamma]^e) \in \calO_S$ for any $\gamma \in \calO$. But $[\gamma]^e$ and $[\alpha\gamma]^e$ are polynomials in $\pi \in \calO_S$, so indeed $\delta \in \calO_S$.
\end{proof}

\begin{remark}
Our groups $\AGL_1(\calO_S)$ are essentially the groups $\PU(2,\calO_S)$ considered by Kochloukova and Sidki in \cite[Theorem~A]{kochloukova17}, and the action is the one obtained by taking $s = [\nu_{t-1}]$. One difference is that they require $[\nu_\infty], [\nu_t] \in S$ while we do not. In particular, all their groups in \cite[Theorem~A]{kochloukova17} are finitely generated. Similarly, if we take $S = \{[\nu_\infty]\}$ then our group $\AGL_1(\calO_S)$ is exactly the group $\FF_p[x] \rtimes B(1,\FF_p[x])$ considered in \cite[Theorem~B]{kochloukova17}.
\end{remark}

We conclude:

\begin{theorem}\label{thm:self_similar_fin_props}
Let $n$ be a positive integer and let $d$ be a prime power. There is a self-similar, finite-state, coarsely-diagonal subgroup of $\Aut(\tree_d)$ that is of type $\F_{n-1}$ but not of type $\FP_n$. If $n = 1$ and $d \ge 3$ then it has finite abelianization.
\end{theorem}

\begin{proof}
Take $k = \FF_d(t)$ and let $s$ be a place of degree $1$. Let $S$ be a set of $n$ places not containing $s$. Then $K \defeq k_s$ is isomorphic as a valued field to $\FF_d\lseries{\pi}$ and in particular $\tree_\calO\cong \tree_d$. The first sequence of claims now follows for $\AGL_1(\calO_S)$ from Proposition~\ref{prop:arith_self_similar} and Theorem~\ref{thm:metabelian_fin_props}.

If $n = 1$ then $\calO_S \cong \FF_{d}[\pi]$, so $G_m(\calO_S)$ is cyclic of finite order. Also, if $d \ge 3$ then $G_m(\calO_S)$ is the abelianization of $\AGL_1(\calO_S)$. Hence the last claim holds.
\end{proof}

\begin{remark}
If $k = \FF_2(t)$ and $S$ contains a single place of degree $1$ then $G_m(\calO_S) = \FF_2^\times$ is trivial and $\AGL_1(\calO_S) = G_a(\calO_S)$ is infinite abelian.
\end{remark}

\section{Proof of the Main Theorem and examples}
\label{sec:proof}

To prove the Main Theorem now we only need to assemble pieces. The precise formulation is as follows.

\begin{theorem}\label{thm:main_thm_precise}
Let $n$ be a positive integer and let $d \ge 3$. If $n = 1$, let $d \ge 4$. There exists a self-similar group $G \le \Aut(\tree_d)$ such that the R\"over--Nekrachevych group $V_d(G)$ is virtually simple and of type $\F_{n-1}$ but not of type $\FP_n$. More specifically, the commutator subgroup $V_d(G)'$ is simple, has finite index in $V_d(G)$, and is of type $\F_{n-1}$ but not of type $\FP_n$.
\end{theorem}

\begin{proof}
Let $d' < d$ be a prime power and if $n = 1$ then take $d'\ge 3$. By Theorem~\ref{thm:self_similar_fin_props} there exists a group $G$ with a faithful, self-similar action on $\tree_{d'}$, such that $G$ is of type $\F_{n-1}$ but not of type $\FP_n$, and the action is finite-state and coarsely diagonal. Applying Lemma~\ref{lem:persist} $d-d'$ times, we get a faithful action of $G$ on $\tree_d$ that in addition to being self-similar, finite-state, and coarsely diagonal is also persistent. We claim that the R\"over--Nekrachevych group $V_d(G)$ is as desired. First, the commutator subgroup $V_d(G)'$ is simple by Theorem~\ref{thm:simple}. It is of type $\F_{n-1}$ by Theorem~\ref{thm:positive} and not of type $\FP_n$ by Theorem~\ref{thm:negative}. Finally, to show that $V_d(G)'$ is of finite index we apply Theorem~\ref{thm:fin_abln} if $n \ge 2$, and Observation~\ref{obs:fin_abln_easy} if $n = 1$.
\end{proof}

We now discuss some concrete examples.

\begin{example}[Grigorchuk group]\label{ex:grig}
Recall the standard generating set for the Grigorchuk group, $\Grig$, is given by $a=(1~2)(\id, \id)$, $b=(a,c)$, $c=(a,d)$, and $d=(\id, b)$. The action on $\tree_2$ is self-similar and finite-state, and is coarsely diagonal since $\Grig$ is torsion. These generators can be extended to create a persistent action on $\tree_3$ via $a=(1~2)(\id, \id, a)$, $b=(a,c,b)$, $c=(a,d,c)$, and $d=(\id,b,d)$. It was shown in \cite{grigorchuk99} that $\Grig$ is not finitely presented. Therefore, although $V_2(\Grig)$ is of type $\F_\infty$ \cite{belk16}, our results show that the virtually simple group $V_3(\Grig)$ is of type $\F_1$ but not of type $\F_2$.
\end{example}

Since the Main Theorem provides the first known examples of finitely presented simple groups that are not of type $\FP_3$, we describe such an example explicitly.

\begin{example}[$\F_2$-not-$\FP_3$]\label{ex:F2_not_FP3}
 Consider the rational function field $\FF_2(t)$. Let
\[
S=\{[\nu_\infty],[\nu_t],[\nu_{1+t+t^2}]\}
\]
and consider the ring of $S$-integers $\calO_S = \FF_2[t,t^{-1},(1+t+t^2)^{-1}]$. Since $|S|=3$, the group $\AGL_1(\calO_S)$ is finitely presented but not of type $\FP_3$. Viewing $\AGL_1(\calO_S)$ as a matrix group, it is
\[
\left\{ \left.\begin{pmatrix} \alpha & \beta\\  & 1 \end{pmatrix} \right| \alpha \in \calO_S^\times, \beta \in \calO_S\right\}\text{.}
\]
A convenient finite generating set is
\[
a = \begin{pmatrix} 1 & 1\\  & 1 \end{pmatrix} \text{, } b = \begin{pmatrix} t & 0\\  & 1 \end{pmatrix} \text{, } c = \begin{pmatrix} 1+t+t^2 & 0\\  & 1 \end{pmatrix}\text{.}
\]
We now consider $\calO_S$ as living inside the completion of $\FF_2(t)$ at the place $[\nu_{1+t}]$, and as in Section~\ref{sec:self_similar_fin_props} we get a self-similar action of $\AGL_1(\calO_S)$ on $\tree_2$. Inspecting the level-$1$ states, as described in Lemma~\ref{lem:arith_states}, and the action on the first level, we see that we have the following wreath recursions:
\[
a = (1~2)(\id,\id) \text{, } b = (b,ab) \text{, } c = (c,bab^{-1}c) \text{.}
\]
It can be checked that the states of $a$ are $\{\id,a\}$, the states of $b$ are $\{b,ab\}$ and the states of $c$ are $\{c,bab^{-1}c,ac,a^{-1}bab^{-1}c\}$, which confirms that the action is finite state. It is also not hard to confirm that it is coarsely diagonal, since the normal closure of $a$ has exponent $2$. Now we extend this to a self-similar action on $\tree_3$ (abusively using the same letters to denote the new elements) via:
\[
a = (1~2)(\id,\id,a) \text{, } b = (b,ab,b) \text{, } c = (c,bab^{-1}c,c) \text{.}
\]
This action of $\AGL_1(\calO_S)$ satisfies all the required conditions, and so the commutator subgroup $V_3(\AGL_1(\calO_S))'$ of the R\"over--Nekrashevych group $V_3(\AGL_1(\calO_S))$ is a finitely presented simple group that is not of type $\FP_3$.

From the recipe in \cite[Theorem~9.14]{nekrashevych04}, one can compute that the abelianization of $V_3(\AGL_1(\calO_S))$ is $\Z/4\Z \oplus \Z/2\Z \oplus \Z/2\Z$ and the abelianization map is described as follows. First let $\chi\colon \AGL_1(\calO_S) \to \Z/4\Z$ be the map $a\mapsto 2+4\Z$, $b\mapsto 1+4\Z$, $c\mapsto 1+4\Z$, let $\chi_b\colon \AGL_1(\calO_S) \to \Z/2\Z$ be the map $a\mapsto 0+2\Z$, $b\mapsto 1+2\Z$, $c\mapsto 0+2\Z$, and let $\chi_c\colon \AGL_1(\calO_S) \to \Z/2\Z$ be the map $a\mapsto 0+2\Z$, $b\mapsto 0+2\Z$, $c\mapsto 1+2\Z$. Then the abelianization map of $V_3(\AGL_1(\calO_S))$ is
\begin{align*}
V_3(\AGL_1(\calO_S)) &\to \Z/4\Z \oplus \Z/2\Z \oplus \Z/2\Z \\
[T_-,\sigma(g_1,\dots,g_n),T_+] &\mapsto (\chi(g_1\cdots g_n),\chi_b(g_1\cdots g_n),\chi_c(g_1\cdots g_n))\text{,}
\end{align*}
and we get a concrete description of $V_3(\AGL_1(\calO_S))'$ as the kernel of this map.
\end{example}

We conclude by distinguishing our examples from Kac--Moody groups. Let $\bfG$ be a Kac--Moody functor and let $\FF_q$ be a finite field. There is an associated twin building $(X_+,X_-)$ and $\bfG(\FF_q)$ acts on $X_+ \times X_-$ as a lattice. Identifying the group with a point orbit is a quasi-isometric embedding by \cite[Theorem~1.1]{caprace09}, so the asymptotic dimension of $\bfG(\FF_q)$ is bounded by the asymptotic dimension of $X_+ \times X_-$ \cite[Proposition~23]{bell08}, which is finite \cite{dymara09}. On the other hand, Lemma~\ref{lem:infinite_rank_abelian_subgroup} implies that all of our examples have infinite asymptotic dimension, cf.\ \cite[Theorems~64,~74]{bell08},
and so none of them is quasi-isometric to any $\bfG(\FF_q)$.

\bibliographystyle{alpha}
\newcommand{\etalchar}[1]{$^{#1}$}

\end{document}